\documentclass[preprint,12pt,article,english]{article}

\usepackage{amssymb}
\usepackage{amsthm}
\usepackage{amsmath}
\usepackage{graphicx}

\newtheorem{thm}{\protect\theoremname}

  \newtheorem{lem}[thm]{\protect\lemmaname}

  \newtheorem{prop}[thm]{\protect\propositionname}

\newtheorem{dfn}[thm]{\protect\definitionname}

\makeatother

\usepackage{babel}
  \providecommand{\lemmaname}{Lemma}
  \providecommand{\propositionname}{Proposition}
\providecommand{\theoremname}{Theorem}
\providecommand{\corollaryname}{Corollary}
\providecommand{\definitionname}{Definition}

\begin{document}

\title{A simple case within Nash-Moser-Ekeland theory\footnote{The work is partially supported by the Scientific Fund of Sofia University under grant 
80-10-133/25.04.2018.}}

\author{M. Ivanov\footnote{Radiant Life Technologies Ltd, www.radiant-life-technologies.com; \emph{milen@radiant-life-technologies.com}} \and N. Zlateva\footnote{Faculty of Mathematics and Informatics, Sofia University,
		5, J. Bourchier blvd., 1164 Sofia, BULGARIA; \emph{zlateva@fmi.uni-sofia.bg}}}
\date{\today}

\maketitle

\begin{abstract}
	We present simple and direct proof to an important case of Nash-Moser-Ekeland theorem.
\end{abstract}



\thispagestyle{empty}
\newpage

\section{Introduction}

In the work \cite{ekeland} Ekeland proved Nash-Moser type theorem (e.g. \cite{hamilton,pdonmt}) for G\^ateaux differentiable function.

Theorem~\ref{thm:ci} below strengthens \cite[Theorem~1]{ekeland} in the sense that all norms are estimated simultaneously. The conditions we impose on the spaces here are more restrictive, but the cases most important for applications are covered.

Here we present a simple and direct proof of this result which in \cite{NME-arxiv} is a corollary to a more general theorem proved for multivalued maps and requiring more sophisticated techniques.

The paper is organised as follows. In Section~\ref{sec:prelim} we recall some definitions and also what we need from the framework of \cite{NME-arxiv}. In Section~\ref{sec: main} we prove the result mentioned above.

\section{Notations and preliminaries}\label{sec:prelim}
Recall that a \emph{Fr\'echet space} is a Hausdorff complete locally convex topological vector space whose topology can be generated by a translation-invariant metric. An equivalent definition involves seminorms. We focus on a special class of Fr\'echet spaces.

We call \emph{compactly graded} space  any $X=\bigcap_0^\infty X_n$, where $(X_n,\|\cdot\|_n)$'s are nested Banach spaces: $X_{n+1}\subset X_n$, such that the identity operator from $X_{n+1}$ into $X_n$ is compact. The most important example of compactly graded space is $X_n=C^n(\Omega)$, where $\Omega$ is a compact domain in $R^k$.

It is immediate to check that such space equipped with  the translation-invariant metric
\begin{equation}\label{eq:rho}
\rho_X(x,y):=\max_{n\ge 0}\frac{2^{-n}\|x-y\|_{n}}{1+\|x-y\|_{n}}
\end{equation}
is a complete metric space. Therefore, any compactly graded space is a Fr\'echet space. Note that
$$
\lim_{k\to\infty}\rho_X(x_k,x)=0\iff \lim_{k\to\infty}\|x_k-x\|_n = 0,\ \forall n\ge 0.
$$
Denote by $B_{X_n}:=\{x\in X_n:\ \|x\|_n\le 1 \}$  the closed unit ball of $(X_n,\|\cdot\|_n)$.

We will now consider the  properties of compactly graded spaces we will be using. First of all, if $J_n:X_{n+1}\to X_n$ is the identity operator, then $\|\cdot\|_n\le \|J_n\|\|\cdot\|_{n+1}$ and, therefore,
\begin{equation}\label{eq:norm-domination}
\|\cdot\|_0\le a_1\|\cdot\|_1\le\ldots\le a_n\|\cdot\|_n\le\ldots,
\end{equation}
where $a_n = \|J_0\|\|J_1\|\ldots\|J_{n-1}\|$.
\begin{prop}\label{pro:comp-crit}
	Let $X=\bigcap_0^\infty (X_n,\|\cdot\|_n)$ be a compactly graded space. Then $C\subset X$ is compact if and only if it is closed and bounded in each norm $\|\cdot\|_n$.
\end{prop}
\begin{proof}
	First, let $C$ be compact. Then it is closed.
	
	If $C$ is unbounded in, say $\|\cdot\|_j$, then the we can chose by induction a sequence $(x_k)_1^\infty\subset C$, such that $\|x_{k+1}\|_j > \max\{\|x_i\|_j:\ i=1,\ldots,k\}$. It is clear that $(x_k)_1^\infty$ has no finite $2^{-j}$-net with respect to $\rho_X$.
	
	Conversely, let  $C$ be closed and
	$$
		s_n := \sup \{ \|x\|_n: x\in C\}<\infty,\quad\forall n\in \mathbb{N}\cup \{0\}.
	$$
	Let $(x_k)_1^\infty\subset C$. Since $(x_k)_1^\infty\subset s_1B_{X_1}$ and the latter is compact in $(X_0,\|\cdot\|_0)$, there exists an infinite set $N_0\subset \mathbb{N}$ such that $\|\cdot\|_0{-}\lim_{k\in N_0} x_k$ exists. By induction we can find infinite sets $N_0\supset N_1\supset\ldots$ such that  the limit $\|\cdot\|_i-\lim_{k\in N_i} x_k$ exists for all $i\ge 0$. Chose $k_0\in N_0$ and 
	$$
		k_i\in N_i:\ k_i > k_{i-1},\quad i\ge 1.
	$$
	Clearly the subsequence $(x_{k_i})_{i=0}^\infty$ is convergent in each norm and, therefore, in $X$.	
\end{proof}
For a compactly graded space $X=\bigcap_0^\infty(X_n,\|\cdot\|_n)$ and $A\subset X$ denote
\begin{equation}\label{def:d-n}
	d_n(x,A) := \inf\{\|x-y\|_n:\ y\in A\}.
\end{equation}
\begin{lem}\label{lem:dn}
	Let $X=\bigcap_0^\infty(X_n,\|\cdot\|_n)$ be a compactly graded space and let $C\subset X$ be compact. Then
	$$
		x\in C\iff d_n(x,C) = 0,\quad \forall n\ge 0.
	$$	
\end{lem}
\begin{proof}
	Fix $\varepsilon>0$ and then take $m$ so large that $2^{-m} < \varepsilon$. 
	
	Set $a:= \max\{a_m,a_m/a_1,\ldots,a_m/a_{m-1}\}$. From \eqref{eq:norm-domination} it follows that
	$$
		\|\cdot\|_n\le a\|\cdot\|_m,\quad\forall n=0,1,\ldots,m.
	$$
	
	Let $y_i\in C$ be such that $\lim_{i\to\infty} \|y_i - x\| _m \to 0$. 
	Since the function $t\to 1/(1+t)$ is increasing, 
	$$
	   \max_{0\le n\le m} \frac{\|y_i-x\|_n}{2^n(1+\|y_i-x\|_n)} \le \frac{a\|y_i-x\|_m}{1+a\|y_i-x\|_m} \to  0, \mbox{ as }i\to \infty.
	$$
	
	From \eqref{eq:rho} and the choice of $m$   it follows that $\rho_X(y_i,x)<2\varepsilon$ for all $i$ large enough.
	
	This means that the $\rho_X${-}distance between $x$ and $C$ is zero. Since $C$ is compact, $x\in C$.	
\end{proof}

\section{Main result}\label{sec: main}

Let us recall that if $(X,\rho_X)$ and $(Y,\rho_Y)$ are linear metric spaces, and
  $$
    f:X\to Y
  $$
  is a function, then the directional derivative of $f$ at $x$ in direction $h$ is
  $$
   f'(x;h):= \rho_Y{-}\lim_{t\searrow 0}\frac{f(x+th)-f(x)}{t}
  $$
  if the latter limit exists.
  
  If the map $h\to f'(x;h)$ is linear and continuous then $f$ is called G\^ateaux differentiable at $x\in X$.
  
  Denote
  $$
  \mathbb{R}_+^\infty := \{(s_{i})_{i=0}^{\infty}:\ s_{i}\ge0\}.
  $$
  For $\mathbf{u}= (u_{i})_{i=0}^{\infty},\mathbf{s}=(s_{i})_{i=0}^{\infty}\in \mathbb{R}_+^\infty$ define
  $$
	  \mathbf{u}.\mathbf{s} := (u_{i}s_i)_{i=0}^{\infty}.
  $$

\begin{dfn} (\cite{NME-arxiv})		
	For $\mathbf{s}=(s_{i})_{i=0}^{\infty}\in \mathbb{R}_+^\infty$ and a compactly graded space $X=\bigcap_0^\infty (X_n,\|\cdot\|_n)$  define
	\begin{equation}
	\label{eq:def:pi}
	\Pi_{\mathbf{s}}(X):=\{x\in X:\ \|x\|_{n}\le s_{n},\ \forall n\ge 0\}.
	\end{equation}
\end{dfn}

It is clear that $\Pi_{\mathbf{s}}(X)$ is closed, so by Proposition~\ref{pro:comp-crit},
\begin{equation}\label{ps-comp}
	\Pi_{\mathbf{s}}(X)\mbox{ is compact, }\forall \mathbf{s} \in \mathbb{R}_+^\infty. 
\end{equation}

We present a simple proof of the following extension of a result form \cite{NME-arxiv}.
	\begin{thm}\label{thm:ci}
		Let $X=\bigcap_0^\infty(X_n,\|\cdot\|_n)$ and $Y=\bigcap_0^\infty (Y_n,|\cdot|_n)$ be compactly graded spaces. Let $f:X\to Y$ be a continuous and directionally differentiable  function such that $f(0)=0$.
		
		Assume that there are $c_n>0$ and $d\in \{0\}\cup\mathbb{N}$, such that for each $x\in X$ and $v\in Y$
	\begin{equation}\label{eq:f-prime:ld}
			\exists u\in X:\ f'(x;u)=v\mbox{ and }
			\|u\|_n\le c_n|v|_{n+d},\quad\forall n\ge 0.
		\end{equation}	
		Then for each  $y\in Y$ there is $x\in X$ such that
		$$
		f(x)=y\mbox{ and }
		\|x\|_n\le c_n|y|_{n+d},\quad\forall n\ge 0.
		$$		
	\end{thm}

\begin{proof}
Since $Y=\bigcap_{n=d}^\infty (Y_n,|\cdot|_n)$, we can assume without loss of generality that $d=0$. Indeed, set $m:=n-d$, $n=d,d+1,\ldots$ and work with the new $m$-indexation in $Y$.

For $A\subset X$ define
$$
	f'(x,A) := \{f'(x;h):\ h\in A\}.
$$	
	With such notation condition \eqref{eq:f-prime:ld} can be rewritten as
	\begin{equation}\label{eq:f-prime:ld:rewritten}
			\Pi_{\mathbf{b}.\mathbf{s}}(Y)\subset f'(x;\Pi_{\mathbf{s}}(X)), \quad\forall x\in X,\ \forall \mathbf{s}\in \mathbb{R}_+^\infty,
	\end{equation}
	where 
	$$
		\mathbf{b} = (c_n^{-1})_{n=0}^\infty.
	$$
	Indeed, if $v\in \Pi_{\mathbf{b}.\mathbf{s}}(Y)$ then $|v|_n\le c_n^{-1}s_n$ and from \eqref{eq:f-prime:ld} it follows that there is $u\in X$ such that $f'(x;u) = v$ and $\|u\|_n\le c_n|v|_n\le s_n$. That is, $u\in \Pi_{\mathbf{s}}$ and $v\in f'(x;\Pi_{\mathbf{s}}(X))$.

We will show that
\begin{equation}\label{eq:conclusion:rewritten}
	\Pi_{\mathbf{b}.\mathbf{s}}(Y)\subset f(\Pi_{\mathbf{s}}(X)),\quad \forall \mathbf{s}\in \mathbb{R}_+^\infty.
\end{equation}

So, fix $ \mathbf{s}\in \mathbb{R}_+^\infty$ and $\bar y \in  \Pi_{\mathbf{b}.\mathbf{s}}(Y)$. 

In view of Lemma~\ref{lem:dn} it is enough to show that $d_n(\bar y, f(\Pi_{\mathbf{s}}(X))) \le \varepsilon$ for each $n\ge 0$ and $\varepsilon>0$, where $d_n$ is defined by \eqref{def:d-n}, of course, in $Y$.
 
Fix arbitrary $n\ge 0$ and $\varepsilon>0$. Define
  $$
    I := \{t\ge 0:\ d_n(r\bar y,f(r\Pi_{\mathbf{s}}(X)) \le r\varepsilon,\quad \forall  r\in[0,t]\},
  $$
 Obviously $0\in I$. Denote $
    \bar t := \sup_{t\in I} t$.

  Suppose that $\bar t < \infty$. Then there are $t_i\nearrow\bar t$ and $x_i\in t_i\Pi_{\mathbf{s}}(X)$, such that
  $$
    \|t_i \bar y - f(x_i)\|_n\le t_i\varepsilon.
  $$

 \begin{figure}[h]
 \centering
 \includegraphics[width=0.8\textwidth]{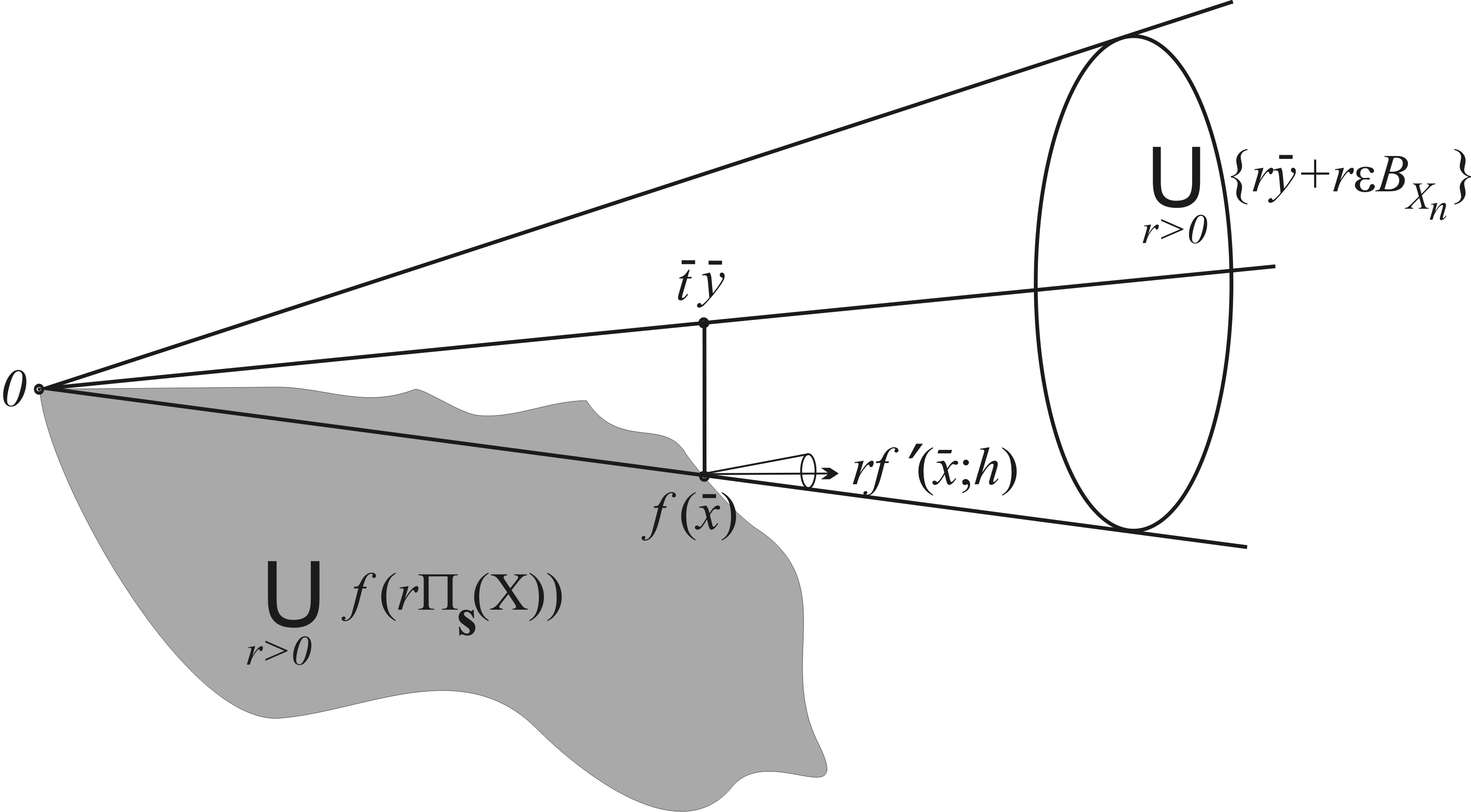}
 \end{figure}

The set  $\bar t \Pi_{\mathbf{s}}(X)$ is compact (see \eqref{ps-comp}) and $x_i\in \bar t \Pi_{\mathbf{s}}(X)$, so we may assume without loss of generality that  $x_i\to\bar x\in\bar t \Pi_{\mathbf{s}}(X)$. Since $f$ is continuous, $f(x_i)\to f(\bar x)$ and, therefore, $\|\bar t \bar y - f(\bar x)\|_n\le \bar t\varepsilon$.

Since  $\bar y\in \Pi_{\mathbf{b}.\mathbf{s}}(Y)\subset f'(\bar x ) (\Pi_{\mathbf{s}}(X))$, see \eqref{eq:f-prime:ld:rewritten}, there is $h\in \Pi_{\mathbf{s}}(X)$ such that $f'(\bar x; h)=\bar y$.

  From the definition of the directional derivative $f'(\bar x; h)$ and \eqref{eq:rho} it follows that
  $$
    0=\lim_{r\searrow 0} \left\|\frac{f(\bar x + rh) -f(\bar x)- rf'(\bar x;h)}{r}\right\|_n = \lim_{r\searrow 0}\frac{\|f(\bar x + rh) -f(\bar x)- r\bar y\|_n}{r} .
  $$  
  Thus for some $\delta>0$ and all $r\in [0,\delta]$ we have that $\|f(\bar x + rh)-f(\bar x)-r\bar y\|_n \le r\varepsilon$. So,
  $\bar x + rh \in (\bar t +r) \Pi_{\mathbf{s}}(X)$ and
  \begin{eqnarray*}
  	d_n((\bar t +r)\bar y,f((\bar t + r)\Pi_{\mathbf{s}}(X))) &\le&
  	\|\bar t\bar y  +r\bar y - f(\bar x + rh)\|_n\\
  	&\le& \|\bar t \bar y - f(\bar x)\|_n + \|f(\bar x + rh)-f(\bar x)-r\bar y\|_n \\
  	&\le& (\bar t + r)\varepsilon.
  \end{eqnarray*}
  That is, $\bar t + \delta\in I$, contradiction. Therefore, $\bar t =\infty$. In particular $1\in I$, thus
  $$
  	d_n(\bar y,f(\Pi_{\mathbf{s}}(X)) \le \varepsilon,
  $$
  yielding \eqref{eq:conclusion:rewritten}, which is easily seen to be equivalent to the conclusion of the theorem. 
\end{proof}

\medskip\noindent

\textbf{Remark.} We have presented Theorem~\ref{thm:ci} in a simple practical context. It can be easily somewhat generalised. For example, we may assume only that $X$ and $Y$ are Fr\'echet spaces with Heine-Borel property, that is, each closed and bounded subset is compact. (The roundedness is understood in the sense of locally convex spaces -- i.e. the set is absorbed by any neighbourhood of zero -- and not with respect to a metric, see e.g. \cite[pp.110-114]{fabetal} for details.)

In such setting $\Pi_{\mathbf{s}}(X)$ is compact for each $\mathbf{s}\in\mathbb{R}_+^\infty$ and this is all we need. The space $Y$ can even be arbitrary Fr\'echet space, see \cite{NME-arxiv}, but the latter extension would require minor modifications to the proof.

\section*{Acknowledgment} Thanks are due to an anonymous referee for numerous remarks considerably improving the quality of this article. From these notable are: pointing out that our definition of differentiability was more restrictive than the standard one, suggesting the name \textit{compactly graded space}, hinting at generalising Theorem~\ref{thm:ci} to sequence of $c$'s.


\begin{thebibliography}{90}
\bibitem{pdonmt}
	A. Alinhac and P. G\'erard, Pseudo-differential Operators and the Nash-Moser Theorem, Graduate Studies in Mathematics, 82, AMS, Providence, RA.
	
\bibitem{ekeland}
	I. Ekeland, An inverse function theorem in Fr\'echet spaces, Ann. Inst. H. Poincar\'e,
	C, 28 (1), 2011, 91--105.
	
\bibitem{fabetal}	
	M. Fabian et al., Banach space theory. The basis for linear and nonlinear analysis, Berlin: Springer, 2011.


\bibitem{hamilton}
	R. S. Hamilton, The Inverse Theorem of Nash and Moser, Bulletin of AMS, 7(1), 1982, 65--222.


\bibitem{NME-arxiv}
	M. Ivanov and N. Zlateva, Surjectivity in Fr\'echet spaces, arXiv preprint arXiv:1805.07055, 2018.

\end{thebibliography}
\end{document}